\documentclass{mathincs}
\usepackage[latin1]{inputenc}
\usepackage[english]{babel}
\usepackage{url}
\usepackage{amsfonts}

\newtheorem{thm}{Theorem}[section]
 \newtheorem{cor}{Corollary}[section]
 
 \newtheorem{prop}{Proposition}[section]
 \newtheorem{rem}{Remark}[section]

 \newcommand{\sep}[1]{{\left({#1}\right)}}
\newcommand{\ent}[1]{{\left[{#1}\right]}}

 \numberwithin{equation}{section}
 
\newcommand\R{\mathbb{R}}

\begin{document}

\title[On Chouikha's isochronicity criterion]
{On Chouikha's isochronicity criterion} 

\author{Jean-Marie Strelcyn}
\address{Laboratoire de Mathématiques Rapha\"el Salem,\\
CNRS, Universit\'e de Rouen \\
Avenue de l'Université BP 12\\
76801 Saint Etienne du Rouveray, \;France}
\email{ Jean-Marie.Strelcyn@univ-rouen.fr\medskip}
\address{ Laboratoire Analyse G\'eometrie et Applications,\\
     UMR CNRS 7539, Institut Gallilée, Université Paris 13,\\
      99 Avenue J.-B. Clément, 93430 Villetaneuse, France}
\email{strelcyn@math.univ-paris13.fr}



\begin{abstract}
Recently A.R.Chouikha gave a new characterization of isochronicity of center at the origin for the equation $x''+g(x)=0$, where $g$ is a real smooth function defined in some neighborhood of $0 \in \R$. We describe another proof of his characterization and some new development of the subject.
\end{abstract}

\maketitle {}

\section{Introduction}

Let us consider the second order differential equation
\begin{equation}\label{(1.1)} 
x''+g(x)=0 
\end{equation}
where $g$ is a real function defined in some neighborhood of $0\in \R$ such that $g(0)=0$. The study of equation \eqref{(1.1)}  is equivalent to the study of planar system

\begin{equation}\label{(1.2)}
  \left. \begin{aligned}\dot x&= y\\
      \dot y&= -g(x)
 \end{aligned}\right\}.
\end{equation}
In what follows we shall concentrate exclusively on the system \eqref{(1.2)} with function $g$ at least of class $C^1$.

As\; $g(0)=0,\; 0\in \R^2$ is a singular point of the system \eqref{(1.2)}. If in some neighborhood of a singular point all orbits of the system are closed and surround it, then the singular point is called a {\it center}.

A center is called {\it isochronous} if the periods of all orbits in some neighborhood of it are constant, that means that the period does not depend on the orbit.
The simplest example of an isochronous center at $0\in \R^2$ is provided by the linear system

\begin{equation}\nonumber
  \left. \begin{aligned}\dot x&= y\\
      \dot y&= -x
\end{aligned}\right\}
\end{equation}
which corresponds to\; $g(x)=x$.

In future when speaking about isochronicity we always understand it with respect to $0\in \R^2$ and the system \eqref{(1.2)}. 

The problem of characterization of isochronicity of the system \eqref{(1.2)} at $0\in \R^2$ in term of function $g$ is an old one.

To the best of our knowledge the first such characterization was done in 1937 by I.Kukles and N.Piskunov in \cite{K-P}, where even the case of continuous functions $g$ is considered. The second one was described in 1962 by M.Urabe in \cite{Urabe1962} (see also \cite{RoRo1999}). Unfortunately these characterizations are not easy to handle and they are not really explicit.

The Urabe criterion plays a key role in the present paper. Thus for the sake of completeness we shall recall it now.

We shall denote 

\begin{equation}\label{(1.3)} 
G(x)=\int_{0}^{x} g(u) \,du
\end{equation}
and call the function $G$ the {\it potential} of $g$. 
 When $0\in \R^2$ is an isochronous center for the system \eqref{(1.2)}, $G$ is called an {\it isochronous potential}.
 
 Let us denote by X the continuous function defined in some neighborhood of $0 \in \R$ by

\begin{equation}\label{(1.4)} 
(X(x))^2=2\,G(x)\, \hbox{ \rm and }\, xX(x)>0 \;for \; x\not=0.
\end{equation} 

Let us formulate now {\it Urabe Isochronicity Criterion}.
\begin{thm}[\cite{Urabe1962}]\label{thm:Urabe}
Let $g$ be a $C^1$ function defined in some neighborhood of \;$0 \in \R$. Let $g(0)=0$ and $g'(0)={\lambda}^2, \lambda>0.$ Then $0\in \R^2$ is an isochronous center for the system \eqref{(1.2)} if and only if
\begin{equation}\label{(1.5)} 
g(x)={\lambda}\frac{X(x)}{1+h(X(x))}
\end{equation}
where the function $X$ is defined by \eqref{(1.4)} and where $h$ is a continuous odd function defined in some neighborhood of $\;0 \in \R$.
\end{thm}

The function $h$ is called {\it Urabe function} of the system \eqref{(1.2)}. 

Let us note that $\omega=\dfrac{2\pi}{\lambda}$ is the period of orbits of the above isochronous center.

Let us stress that from \eqref{(1.3)} and from assumptions on $g$ in Urabe theorem it follows that $G(0)=0$ and that in some punctured neighborhood of $0$, $G(x)>0$, $G$ is of class $C^2$. As will be proven in Sec. 2, under our assumptions, $X$ and $h$ are of class $C^1$. In fact, from the proof of Sec. 2, it follows that if $g \in C^k, k \geq 1$ (resp. $g$ is real-analytic), then $X$ and $h$ are of class $C^k$
(resp. $X$ and $h$ are real-analytic). 

From now on we shall always assume that $g \in C^1(]-\epsilon, \epsilon[))$ for some $\epsilon>0$ and that 
\begin{equation*} 
g'(0)=\lambda^2, \;\lambda>0. 
\end{equation*} 

In September 2011 in a highly important paper \cite{cho2011}, A.R.Chouikha published a completely new criterion of isochronicity (\cite{cho2011}, Theorem B) which is much more direct and explicit that all previously known.

\begin{thm}[\cite{cho2011}]\label{thm:Chouikha}
Let $g \in C^1(]-\epsilon, \epsilon[)$ for some $\epsilon>0$. Let $g(0)=0$ and $g'(0)>0$. Then $0\in \R^2$ is an isochronous center for the system \eqref{(1.2)} if and only if there 
exists $\delta$, $0 < \delta \leq \epsilon$ such that for $\vert x \vert \leq \delta$ one has 

\begin{equation}\label{(1.7)} 
\frac{d}{dx}\ent{\frac{G(x)}{g^2(x)}}=f(G(x))
\end{equation}
where $f$ is  a continuous functions defined on some interval $[\,0, \eta\,]$ for some $\eta>0$
\end{thm}

We shall call the equation \eqref{(1.7)} the {\it Chouikha equation}.

Let us pause now in the history of this theorem. 
In early February 2010, A.R. Chouikha communicated to me his {\it first} proof of his theorem valid only  in real-analytic setting. 
Some time after he presented to me the {\it second} proof also valid only in real-analytic setting. 
The first proof was based on Urabe theorem, the second one on  S.N. Chow and D. Wang \cite{Cho-Wan1986} formula for the derivative of the first return map for the system \eqref{(1.2)}.
These proofs were never published. 
At the beginning of July 2011, A.R. Chouikha and myself, simultaneously and independently obtained two different proofs of Chouikha theorem in smooth setting. 
Both proofs are the adaptation of the previous Chouikha's proofs in real-analytic setting. 
The Chouikha's proof published in \cite{cho2011} is the adaptation of his second proof. My proof is the adaptation of his first proof. 

As a consequence of this last proof we obtain an unexpected closed relation between Urabe function $h$ and function $f$ from Chouikha

\begin{cor}
\begin{equation}\label{(1.8)} 
h(s)=\lambda\int_{0}^{s} f({\dfrac{q^2}{2}})dq,
\end{equation}
where $g'(0)=\lambda^2, \;\lambda>0$.
\end{cor}

Let us also note that both proofs together give a short, simple and very clever proof of the Urabe theorem.

The paper is organized as follows. In Sec. 2 we present our proof of Chouikha theorem and obtain the identity \eqref{(1.8)}. In Sec. 3 we prove the existence and uniqueness of solutions for Chouikha equation \eqref{(1.7)}. This gives a new insight on the problem of global parametrization of isochronous center of the system \eqref{(1.2)}, the problem already studied in Sec. 4 of \cite{cho2011}. This is the subject of Sec. 4.

In the whole paper we shall freely use the notations introduced above. Finally, let us stress the completely elementary nature of the present paper.


\section{A proof of Chouikha theorem}

We shall prove that the Chouikha equation \eqref{(1.7)} 

\begin{equation*}
\sep{\frac{G(x)}{g^2(x)}}'=f(G(x)),
\end{equation*}
where $'=\dfrac{d}{dx}$ and where $f$ is  a continuous functions defined on some interval $[\,0, \eta\,]$ for some $\eta>0$, is equivalent to the Urabe condition of isochronicity \eqref{(1.5)}

\begin{equation*}
g(x)={\lambda}\frac{X(x)}{1+h(X(x))}
\end{equation*}
where the function $X$ is defined by \eqref{(1.4)} and where $h$ is a continuous odd function defined in some neighborhood of $\;0 \in \R$.

From \eqref{(1.4)} one obtains that the Chouikha equation \eqref{(1.7)} is equivalent to 

\begin{equation*}
\sep{\frac{(X(x))^2}{2(g(x))^2}}' =\frac{1}{2}\sep{\sep{\frac{X(x)}{g(x)}}^2}'=f\sep{\frac{{X^2}(x)}{2}}
\end{equation*}
and consequently equivalent to
\begin{equation*}
\frac{X(x)}{g(x)}\sep{\frac{X(x)}{g(x)}}^\prime =f\sep{\frac{X^2(x)}{2}}.
\end{equation*}
But from \eqref{(1.4)}, $g(x) = G^\prime(x) = X(x)X^\prime(x)$ and finally the Chouikha equation \eqref{(1.7)} is equivalent to
\begin{equation} \label{(2.1)}
\sep{\frac{X(x)}{g(x)}}^\prime =f\sep{\frac{X^2(x)}{2}}X^\prime(x).
\end{equation}
\vspace{0,2cm}

To integrate the equation \eqref{(2.1)} let us first prove that

\begin{equation}\label{(2.2)} 
\lim_{x \to 0}{\frac{X(x)}{g(x)}}=\frac{1}{\lambda}>0.
\end{equation}
Indeed, as $g'(0)={\lambda}^2, \lambda>0$, then 
\begin{equation}\label{(2.3)} 
g(x)={{\lambda}^2}x+\widetilde{g}(x) , \hbox{ \rm where } \,\widetilde{g}(x)=o(\vert x \vert)
\end{equation}
and
\begin{equation}\label{(2.4)} 
G(x)={\lambda}^2\frac{x^2}{2}+\widetilde{G}(x), \hbox{ \rm where } \,\widetilde{G}(x)=o({\vert x \vert}^2).
\end{equation}
From \eqref{(1.4)} and \eqref{(2.4)} one has
\begin{equation}\label{(2.5)} 
X(x)=x\sqrt{{\lambda}^2 + \frac{2}{x^2}\widetilde{G}(x)}.
\end{equation}
Replacing in the left side of \eqref{(2.2)}  $X(x)$ and $g(x)$ by \eqref{(2.5)} and \eqref{(2.3)} respectively  one obtains the indicated limit.
\vspace{0,1cm}

Now taking into account \eqref{(2.2)} we can integrate the equation \eqref{(2.1)}. Indeed, by the change of variables $X(s)=q$, from \eqref{(2.1)} one obtains:
\begin{equation}\label{(2.6)}
{\frac{X(x)}{g(x)}}=\frac{1}{\lambda}+\int_{0}^{x} f\sep{\dfrac{{X^2}(s)}{2}}X'(s)ds=
\frac{1}{\lambda}+\int_{0}^{X(x)} f\sep{\dfrac{{q^2}}{2}}dq.
\end{equation}

Up to now we have proven the equivalence between Chouikha equation \eqref{(1.7)} and the equation \eqref{(2.6)}.

Let us denote
\begin{equation*}
h(X)=\lambda\int_{0}^{X} f({\frac{q^2}{2}})dq.
\end{equation*}

\noindent $h$ is an odd function as a primitive of an even one. $h$ is of class $C^1$ because $h'(X)={\lambda}f({\dfrac{X^2}{2}})$ and $f$ is continuous; $h$ and $f$ defined for $\vert x \vert$ small enough. Thus for $\vert x \vert$ small enough one has
$${\frac{X(x)}{g(x)}}=\frac{1}{\lambda}\sep{1+h(X(x))}$$ 
and finally $g(x)={\lambda}\dfrac{X(x)}{1+h(X(x))}$, that is the Urabe condition \eqref{(1.5)}.
This concludes the proof that \eqref{(1.7)} implies \eqref{(1.5)} as well as the proof of the formula \eqref{(1.8)}.
\vspace{0,1cm}

Reciprocally, let us note that \eqref{(2.4)} and \eqref{(2.5)} imply that the function $X$ is of class $C^1$ and that $X'(0)=\lambda>0$. Then the inverse function $x=x(X)$ is also of class $C^1$. As a consequence, from \eqref{(1.5)}
one deduces immediately that the odd function $h=h(X)$ is also of class $C^1$. Thus the even function $h'$ is continuous and we define the continuous function $f$ by $h'(X)=\lambda f({\dfrac{X^2}{2}})$. That means that
\begin{equation}\label{(2.7a)}
f(s)=\dfrac{1}{\lambda}h'(\sqrt{2s}), s \geq 0.
\end{equation}
The equivalence between Chouikha equation \eqref{(1.7)} and the equation \eqref{(2.6)} closes the proof.
\vspace{0,1cm}
$\Box$
\vspace{0,1cm}

Let us suppose that we only know that $g'(0)>0$, so that $g'(0)=\nu^2, \nu>0$, and that \eqref{(1.5)} is satisfied for some $\lambda>0$. 
Inspecting the arguments in the proof of the smoothness of functions $X$ and $h$,  one concludes that $X'(0)= \nu$. Now, deriving in $x$ both sides of \eqref{(1.5)} and taking $x=0$ one concludes that $g'(0)={\lambda}{\nu}$. But $g'(0)=\nu^2$ and thus $\nu=\lambda$. This proves 
that in the formulation of Urabe Isochronicity Criterion it suffices to suppose that $g'(0)>0$ and that \eqref{(1.5)}  is satisfied for some 
$\lambda>0$. 
 
Let us verify the formula \eqref{(1.8)} on one concrete example. Let us consider the {\it Urabe isochronous potential}
\begin{equation}\label{(2.7)}
G(x)=\frac{1}{2}\sep{\frac{\sqrt{1+2ax}-1}{a}}^2,
\end{equation}
where $a>0$ and $-\dfrac{1}{2a}<x<\dfrac{1}{2a}$ (\cite{Dor2005}, Sec.C, Family I). For this potential 
\begin{equation*} 
g(x)=\frac{\sqrt{1+2ax}-1}{a\sqrt{1+2ax}},
\end{equation*}
$g'(0)=1$ and $\sep{\dfrac{G}{g^2}}'=const=a$. That means that for Urabe potential $f(G)=const=a$ and then from the formula \eqref{(1.8)} the Urabe function of this potential is equal to $h(X)=aX$ because $\lambda=\sqrt{g'(0})=1$. From \eqref{(2.7)} it is easy to see that for every ${\vert x\vert}<\dfrac{1}{2a}$ one has 
\begin{equation*} 
X(x)=\frac{\sqrt{1+2ax}-1}{a},
\end{equation*}
where $X(x)$ is defined by \eqref{(1.4)}. By direct computation one verifies that $g(x)=\dfrac{X(x)}{1+aX(x)}$ as expected by \eqref{(1.5)}. 


\section{Existence and uniqueness of solutions for Chouikha equation}

Let us denote
\begin{equation*}
F(x)=\int_{0}^{x} f(u) \,du.
\end{equation*}

\noindent $F\in C^2([0, \alpha]), F(0)=0$.\; Moreover, $G\in C^2(]-\delta, \delta[), G(0)=G'(0)=g(0)=0.$

First, let us prove the Proposition that appears also in [Cho], Corollary 3-3.

\begin{prop}
The Chouikha equation \eqref{(1.7)} is equivalent to equation
\begin{equation}\label{(3.4)} 
2G(x)=\ent{x+F(G(x))}G'(x).
\end{equation}
\end{prop}

\begin{proof}
The equation \eqref{(1.7)} can be written as $\sep{\dfrac{G}{g^2}}'=F'(G)$ which means
\begin{equation*}
1-\dfrac{2GG''}{(G')^2}=F'(G)G'=\ent{F(G)}'.
\end{equation*}
On the other hand $\sep{\dfrac{G}{G'}}'=1-\dfrac{GG''}{(G')^2}$. Thus the equation \eqref{(1.7)} is equivalent to 
\begin{equation*}
2\sep {\dfrac{G}{G'}}'=1+\ent{F(G)}'=\ent{x+F(G(x))}'.
\end{equation*}
Integrating the last identity and taking into account that $G'(0)=g(0)\neq0$ and that $G(0)=F(0)=0$ we conclude its equivalence with \eqref{(3.4)}.
\end{proof}

From  now on we shall suppose that $f\in C^1\sep{[0, \epsilon]}, \epsilon>0$, where in $0$ and in $\epsilon$ one considers the one-sided first derivatives. This condition can surely be weakened but we shall not examine this problem here. As before $g\in C^1\sep{]-\delta, \delta[},\; \delta>0$. 

Let us recall the well known {\it Hadamard lemma} (\cite{Arn1992}, p.143 and \cite{Hartman2002}, Chap. V, Lemma 3.1) that we formulate for one variable case only.

\vspace{0,1cm}
\noindent {\bf Hadamard lemma.} \; {\it Let $P \in C^k(]-\alpha, \alpha[)$ \; for some $\alpha>0$ where $k \geq 1$ and let $P(0)=0$. Then for $|x|<\alpha$ one has $P(x)=xQ(x)$ where $Q \in C^{k-1}(]-\alpha, \alpha[)$}.
\vspace{0,1cm}

Its proof works as well if instead of open interval $]-\alpha, \alpha[$ one takes the half-open interval 
$[0, \alpha[$. We shall use this remark afterwards. 
\vspace{0,1cm}

By Hadamard lemma one has
\begin{equation}\label{(3.3)}
F(x)=xK(x) \;\;\hbox{ \rm and } \;\; G(x)=x^2H(x),
\end{equation}where $K\in C^1([0, \epsilon]), \epsilon>0$ and $H\in C^0(]-\delta, \delta[),\; \delta>0$. Moreover, 
$H\in C^2(]-\delta,\delta[) \; \backslash \; \{0\})$.

Let us note that for $x \neq 0, \vert x \vert < \delta$, the equation \eqref{(3.4)} is equivalent to the equation
\begin{equation*}
H'(x)=-\dfrac{2H^2(x)K(x^2H(x))}{1+xH(x)K(x^2H(x))}.
\end{equation*}
This remark is at the basis of the proof of the following theorem.

\begin{thm}\label{(thm 3.1)}
 Let $\epsilon>0$ and $\lambda>0$. Let $f\in C^1([0, \epsilon])$. There exists $\delta,\; 0<\delta\leq\epsilon$ and a unique function 
$g\in C^1(]-\delta, \delta[)$, $g'(0)={\lambda}^2$ such that for every $\vert x \vert<\delta$ the Chouikha equation  \eqref{(1.7)}
\begin{equation*} 
\frac{d}{dx}\ent{\frac{G(x)}{g^2(x)}}=f(G(x)) 
\end{equation*} 
is satisfied.   
\end{thm}

\begin{proof}
First we shall prove that if $G$ satisfies the Chouikha equations, then the function $H$ defined by \eqref{(3.3)} is of class $C^2$ in some neighborhood of $0$.

Let $a \in \R$ be fixed. Let $\eta>0$ be so small that $1+HK(x^2H)>0$ for 
\begin{equation*}
(x,H) \in \,]-\eta, \eta[ \times ]a-\eta, a+\eta[ \;:=U(a, \eta).
\end{equation*}

\noindent For such $(x,H)$ consider the function of two variables

\begin{equation*}
\Phi(x,H)=-\dfrac{2H^2K(x^2H)}{1+xHK(x^2H)}.
\end{equation*}

\noindent As $K$ is of class $C^1$, $\Phi \in C^1(U(a,\eta))$.

Now let us consider the Cauchy problem
\begin{equation}\label{(3.5)}
H'(x)=\Phi(x,H(x)),\;  H(0)=a.
\end{equation}
This problem has the unique solution  $H\in C^2(]-\mu, \mu[)$ for some $\mu>0$ (see \cite{Arn1992}, Sec. 32). 
The condition $g'(0)={\lambda}^2$ now reads  
\begin{equation}\label{(3.6)}
H(0)=\dfrac{{\lambda}^2}{2}.
\end{equation}
Now let us take $a=\dfrac{{\lambda}^2}{2}$. Let us consider the solution $H$ of Cauchy problem \eqref{(3.5)} for such $a$. 
Then $G(x)=x^2H(x), \vert x \vert<\mu,$ is the unique $C^2$ solution of equation \eqref{(3.4)} satisfying \eqref{(3.6)} and consequently
$g(x)=G'(x)=\sep{x^2H(x)}'$ is a unique $C^1$ solution of Chouikha equation satisfying $g'(0)={\lambda}^2$.
\end{proof}

Let us stress that if $f_1, f_2 \in C^1([0, \epsilon]), \epsilon>0$, and $f_1 \neq f_2$ on every interval $[0, \eta],\; 0<\eta\leq\epsilon$, then in any neighborhood of \,$0 \in \R, \, g_1 \neq g_2$, where $g_1$ and $g_2$ are the solutions of Chouikha equation that correspond to $f_1$ and to $f_2$ respectively.

Let us  also note that if one supposes that $f\in C^k[0, \epsilon]), 1 \leq k \leq \infty$, or\, f\, is real-analytic, then the unique solution $g$ of Chouikha equation is also of the same class (see \cite{Arn1992}, Sec. 32 for smooth case, and \cite{Hoch1975}, Sec. 6.2 for real-analytic case).

As a consequence of Theorem \ref{(thm 3.1)} and of \eqref{(1.8)} we obtain a fact that seems to have been completely overlooked until now.

\begin{rem} \label{remark} To every odd function $h$ of class $C^k, 1 \leq k \leq \infty$ (resp. real-analytic) defined in some neighborhood of \;  $0 \in \R$ and to every real number $\lambda>0$ there corresponds a unique function g of class $C^{k-1}$ (resp. real-analytic) defined in some neighborhood of \; $0 \in \R,\; g(0)=0,\; g'(0)=\lambda^2$ such that $0\in \R^2$ is an isochronous center for the system \eqref{(1.2)}  and that $h$ is its Urabe function.
\end{rem}

\section{Parametrization of isochronous centers}

From now on, up to the end of this Section we shall only consider the case of real-analytic or $C^\infty$ functions $g$. Let us suppose that for  function $g$, $0\in \R^2$ is an isochronous center for the system \eqref{(1.2)}. 

In the real-analytic case the corresponding Urabe function $h$, which is uniquely determined by $g$, is also real-analytic. Indeed, in this case the formula \eqref{(2.5)} proves that $X$ is a real-analytic function of $x$ such that $X'(0)=\lambda>0$. Thus the inverse function $x=x(X)$ is also real-analytic. Now, the real analyticity of $h$ follows from \eqref{(1.5)}. $h$ being odd, all powers from its Taylor expansion are odd;
\begin{equation}\label{(3.7)}
h(s)=\sum_{k\geq 0} h_{2k+1}s^{2k+1}.
\end{equation}
Deriving the identity \eqref{(1.8)} one obtains
\begin{equation}\label{(3.8)}
h'(s)=\sqrt{g'(0)}f(\frac{s^2}{2}).
\end{equation}
Now, from \eqref{(3.7)} and \eqref{(3.8)} one easily deduces the Taylor development of function $f$
\begin{equation*}
f(s)=\dfrac{1}{\sqrt{g'(0)}}\sum_{k\geq 0}
(2k+1)2^k{h_{2k+1}}s^k.
\end{equation*}
Thus the function $f$, the existence of which is given by Chouikha theorem, is real-analytic.

In the $C^\infty$ case, the same arguments as above proves that the function $h$ is also of class $C^\infty$. The function $h'$ being even, from formula 
\eqref{(2.7a)} together with Taylor formula with remainder term of arbitrary high order applied to $h'$, one concludes that the function $f$ is also of class $C^\infty$.

Together with the real-analytic counterpart of Theorem \ref{(thm 3.1)} this proves that in the real-analytic case
there exists a natural bijective correspondence between the set of the couples of real-analytic functions $f$ defined in some neighborhood of $0 \in \R$ and of real numbers $\lambda>0$ and the set of  the real-analytic functions $g$ such that $0\in \R^2$ is an isochronous center for the system \eqref{(1.2)}. Indeed, to real-analytic function $f$ defined in some neighborhood of $0 \in \R$ and to real number $\lambda>0$ we associate the unique real-analytic function $g$ such that $g(0)=0,\; g'(0)={\lambda}^2$ which is a solution of Chouikha equation, the existence of which  is given by Theorem \ref{(thm 3.1)}.

Knowing already that the function $h$ is of class $C^\infty$, applying the $C^\infty$ counterpart of Theorem \ref{(thm 3.1)} one proves the same statement for $C^\infty$ functions $f$ defined on some interval $[0, \epsilon],\; \epsilon>0$.

Let us denote by $Isochr(0,\omega)$ the germs of isochronous centers of the equation $x''+g(x)=0$ where g is a
real-analytic function defined in some neighborhood of $0\in \R, \, g(0)=0, \, g'(0)>0$. Let us denote by $C_{0}^{\omega}$ the germs 
of real-analytic functions defined in some neighborhood of $0 \in \R$. We can then state:

\begin{thm}\label{(thm 3.2)}
The Cartesian product $C_{0}^{\omega} \times  \{x \in \R; x>0\}$ and the set $Isochr(0,\omega)$ are in natural bijective correspondence. In other words the germs of real-analytic functions defined in some neighborhood of $0 \in \R$ and the strictly positive real numbers parametrize the germs of isochronous centers at $0$ of equation $x''+g(x)=0$,  with $g$ real-analytic function defined in some neighborhood of \; $0 \in \R, \; g(0)=0, \; g'(0)>0$. 
\end{thm}

The above statement gives a new light on the matter of Sec.4 of \cite{cho2011}, especially about the convergence of power series 
which appear  there.

Let us stress that from the preceding considerations it follows that the completely analogous statement to Theorem \ref{(thm 3.2)} is valid also in $C^\infty$ framework.

Why did we not formulate the similar statement in classes $C^k, k \geq 1$? Actually, the formula \eqref{(1.8)} implies the loss of differentiability when passing from function $h$ to function $f$ (compare Remark \ref{remark}).

Finally let us note that for a given function $f$ and given $\lambda>0$ the well known numerical methods applied to the Cauchy problem 
\begin{equation*}
H'(x)=\Phi(x,H(x)),\;  H(0)=\dfrac{{\lambda}^2}{2}
\end{equation*}
make it possible to find the corresponding isochronous potential $G$ with arbitrary high precision.


\section*{Acknowledgment}
I thank A.Raouf Chouikha (University Paris 13) for many years of helpful discussions on isochronous centers. I thank also Andrzej Maciejewski and Maria Przybylska (both from University of Zielona Gora, Poland) and Alain Albouy (Observatoire de Paris) for interesting discussions. In particular Andrzej Maciejewski noted that the use of Hadamard Lemma can replace my more tedious arguments and Alain Albouy critically read the  previous version of this paper. Last but not least I thank Marie-Claude Werquin (University Paris 13) for linguistic corrections.
\bibliographystyle{plain}
\bibliography{BBCS}

\end{document}